\documentclass[10pt,letterpaper]{article}

\usepackage[T1]{fontenc}
\usepackage[utf8]{inputenc}
\usepackage[margin=1in]{geometry}
\usepackage{amsmath,amssymb,amsthm,mathtools}
\usepackage{microtype}
\usepackage{enumitem}
\usepackage{xcolor}
\usepackage{aliascnt}
\usepackage[colorlinks=true,linkcolor=blue!55!black,citecolor=blue!55!black,urlcolor=blue!55!black]{hyperref}
\hypersetup{pdftitle={An improved finite bound for oriented trees in tournaments},pdfauthor={Jiangdong Ai}}
\usepackage[nameinlink,capitalise,noabbrev]{cleveref}

\setlist[itemize]{leftmargin=1.7em,itemsep=2pt,topsep=4pt}
\setlist[enumerate]{leftmargin=1.9em,itemsep=2pt,topsep=4pt}

\newtheorem{theorem}{Theorem}
\newaliascnt{lemma}{theorem}
\newtheorem{lemma}[lemma]{Lemma}
\aliascntresetthe{lemma}
\newaliascnt{proposition}{theorem}
\newtheorem{proposition}[proposition]{Proposition}
\aliascntresetthe{proposition}
\newaliascnt{corollary}{theorem}
\newtheorem{corollary}[corollary]{Corollary}
\aliascntresetthe{corollary}
\theoremstyle{definition}
\newaliascnt{definition}{theorem}

\aliascntresetthe{definition}
\theoremstyle{remark}
\newaliascnt{remark}{theorem}
\newtheorem{remark}[remark]{Remark}
\aliascntresetthe{remark}

\crefname{theorem}{Theorem}{Theorems}
\crefname{lemma}{Lemma}{Lemmas}
\crefname{proposition}{Proposition}{Propositions}
\crefname{corollary}{Corollary}{Corollaries}
\crefname{definition}{Definition}{Definitions}
\crefname{remark}{Remark}{Remarks}

\newcommand{\unv}{\operatorname{unvd}}
\newcommand{\cdown}{\mathcal C^{\downarrow}}
\newcommand{\cupward}{\mathcal C^{\uparrow}}
\newcommand{\gdown}{\gamma^{\downarrow}}
\newcommand{\gup}{\gamma^{\uparrow}}
\newcommand{\Lminus}{L^{-}}
\newcommand{\Lplus}{L^{+}}
\newcommand{\Nminus}{N^{-}}
\newcommand{\Nplus}{N^{+}}

\title{An improved finite bound for oriented trees in tournaments}
\author{Jiangdong Ai\thanks{School of Mathematical Sciences and LPMC, Nankai University, Tianjin 300071, P.R. China. Email: \href{mailto:jd@nankai.edu.cn}{\texttt{jd@nankai.edu.cn}}.},~ 
Xiaopan Lian\thanks{ Center for Combinatorics and LPMC, Nankai University, Tianjin 300071, P.R.
  China. Email: Lian@nankai.edu.cn.}}
\date{}

\begin{document}
\maketitle

\begin{abstract}
Sumner's universal tournament conjecture asserts that every tournament on $2n-2$ vertices contains every oriented tree on $n$ vertices. Let $f(n)$ be the least integer $N$ such that every tournament on $N$ vertices contains every oriented tree on $n$ vertices. Havet and Thomass\'e proved that $f(n)\le \lceil(7n-5)/2\rceil$, El Sahili improved this to $f(n)\le3n-3$, and Dross and Havet subsequently obtained $f(n)\le\lceil21n/8-47/16\rceil$. We refine their median-order method. More precisely, every non-bi-arborescence on $n$ vertices with $k$ leaves is $(4n-2k-4)$-unavoidable, which strictly improves their many-leaf estimate; bi-arborescences satisfy the stronger bound $2n-2$. Combining this refinement with their few-leaf bound gives $f(n)\le\lceil(18n-23)/7\rceil$ for every $n\ge2$. Thus the coefficient in the previously best general bound valid uniformly for all $n$ is reduced from $21/8$ to $18/7$.
\end{abstract}

\section{Introduction}

A tournament is an orientation of a complete graph. An oriented graph $D$ is $m$-\emph{unavoidable} if every tournament on $m$ vertices contains a copy of $D$, and $\unv(D)$ denotes the least such $m$. Write $f(n)=\max\{\unv(A):A\text{ is an oriented tree on }n\text{ vertices}\}$. Sumner's universal tournament conjecture states that $f(n)=2n-2$ for every $n\ge2$. The bound would be best possible, as shown by an out-star and a regular tournament on $2n-3$ vertices. K\"uhn, Mycroft and Osthus proved the conjecture for all sufficiently large $n$~\cite{KuhnMycroftOsthus2011}; here the issue is the strongest bound valid uniformly for every order.

The first linear estimate was obtained by H\"aggkvist and Thomason~\cite{HaggkvistThomason1991}, followed by improvements of Havet~\cite{Havet2002}, Havet and Thomass\'e~\cite{HavetThomasse2000}, and El Sahili~\cite{ElSahili2004}. Dross and Havet proved that every oriented tree on $n$ vertices with $k$ leaves is both $\lceil\frac32(n+k)-2\rceil$-unavoidable and, for $n\ge3$, $\lceil\frac92n-\frac52k-\frac92\rceil$-unavoidable. Their two estimates yield $f(n)\le\lceil\frac{21}{8}n-\frac{47}{16}\rceil$ for all $n\ge2$~\cite{DrossHavet2021}. Later work gives substantially sharper results in asymptotic or leaf-restricted regimes~\cite{BenfordMontgomery2022,BenfordMontgomery2025}, while recent related work still records the Dross--Havet estimate as the standard general finite bound~\cite{AboulkerEtAl2024}.

Our main result is the following.

\begin{theorem}\label{thm:main}
Every oriented tree $A$ on $n\ge2$ vertices satisfies $\unv(A)\le\lceil(18n-23)/7\rceil$.
\end{theorem}

For comparison, the values of the previous and new bounds for the first few orders are
\[
\begin{array}{c|rrrrrrr}
 n&2&3&4&5&6&7&8\\ \hline
 \text{Dross--Havet}&3&5&8&11&13&16&19\\
 \text{\cref{thm:main}}&2&5&7&10&13&15&18.
\end{array}
\]
The proof also yields a parameterised improvement: every non-bi-arborescence with $n$ vertices and $k$ leaves is $(4n-2k-4)$-unavoidable. Since such a tree has $k\le n-2$, this is at least one integer below the Dross--Havet many-leaf bound throughout its range.

The argument refines the three-phase median-order embedding in the proof of the many-leaf estimate of Dross and Havet. Components directed towards a chosen root are flattened, and artificial children are introduced to form an equivalent out-arborescence. The greedy embedding of this arborescence leaves one failed position fewer than its number of out-leaves. The original many-leaf count uses the failed positions paired with the original heart leaves. We show that the same online pairing applies after the child identities have been chosen and hence also to the artificial leaves. Each artificial leaf therefore provides both a released image and, with one global exception, a later failed position. This shortens the buffer by precisely the number of artificial leaves.

If the two leaf clusters together have $s$ vertices and the remaining heart has $h$ vertices, the resulting estimate is $\unv(A)\le2s+4h-5=2n+2h-5$ for $h\ge3$. Since all leaves of $A$ lie in the clusters, the few-leaf theorem gives $\unv(A)\le\lceil3n-\frac32h-2\rceil$. The two bounds meet at $h=2(n+3)/7$, which gives \cref{thm:main}. We make the greedy arborescence argument and its failed-position injection self-contained in \cref{sec:prelim}; the refined recovery is proved in \cref{sec:recovery}, and the global bounds are derived in \cref{sec:global}. Two additional refinements are recorded in \cref{sec:sharing,sec:brooms}.

\section{Preliminaries}\label{sec:prelim}

All digraphs are finite and have neither loops nor parallel arcs. For a vertex $v$ of a digraph $D$, write $\Nplus_D(v)$ and $\Nminus_D(v)$ for its out- and in-neighbourhoods. A leaf of an oriented tree $A$ is an \emph{in-leaf} if it has indegree $0$ and outdegree $1$, and an \emph{out-leaf} if it has indegree $1$ and outdegree $0$. We write $\Lminus(A)$, $\Lplus(A)$, and $L(A)$ for the corresponding sets.

An ordering $\sigma=(v_1,\ldots,v_m)$ of a tournament $T$ is a \emph{median order} if it maximises the number of forward arcs. It is a \emph{local median order} if, for every $1\le i<j\le m$,
\begin{align}
 |\Nplus_T(v_i)\cap\{v_{i+1},\ldots,v_j\}|&\ge \frac{j-i}{2},\label{eq:feedback-forward}\\
 |\Nminus_T(v_j)\cap\{v_i,\ldots,v_{j-1}\}|&\ge \frac{j-i}{2}.\label{eq:feedback-backward}
\end{align}
Every median order is local, every interval of a median order is a median order of the induced subtournament, and every interval of a local median order is local; see~\cite{HavetThomasse2000,DrossHavet2021}. A \emph{final interval} of an order is an interval containing its last position, and it is \emph{proper} if it is not the whole order. If $\tau=(v_a,\ldots,v_b)$ is an interval of a local median order and $D$ is a rooted out-arborescence whose root is mapped to $v_a$, we call an embedding $\phi$ \emph{root-forward in $\tau$} if every proper final interval $J$ of $\tau$ satisfies $|\phi(V(D))\cap J|<|J|/2+1$. This is the exact invariant needed when the ambient interval has length $2|D|-2$.

The next lemma is the greedy argument of Dross and Havet~\cite[Theorem~12 and Observation~13]{DrossHavet2021}. We include the proof because its online form is the point at which the artificial leaves enter our argument.

\begin{lemma}\label{lem:greedy-arborescence}
Let $B$ be an out-arborescence with $n_B\ge2$ vertices, $k_B$ out-leaves, and root $r$, and let $\sigma=(v_1,\ldots,v_{n_B+k_B-1})$ be a local median order of a tournament. The greedy procedure that sends $r$ to $v_1$, scans the positions from left to right, and at the image of a vertex assigns as many of its unembedded children as possible, in order, to the first currently available out-neighbours to its right, succeeds. It leaves exactly $k_B-1$ failed positions. Moreover, there is an injection from the failed positions to $\Lplus(B)$ such that the image of the assigned leaf precedes the corresponding failed position. Consequently, all but at most one out-leaf have distinct failed mates after their images.

The identities and order of the children of a vertex may be chosen after their complete image set has been exposed, provided the choice is made before any position in that image set is scanned.
\end{lemma}

\begin{proof}
Call a position \emph{hit} once it is assigned to a vertex of $B$, and \emph{failed} if it is unhit when scanned. When a hit position $v_i$ is scanned, assign as many unembedded children of its preimage as possible to the first currently unhit out-neighbours of $v_i$ to its right. The state \emph{at} $i$ means the state immediately before $v_i$ is scanned. Let $F$ be the set of failed positions in the completed execution, let $W$ be the set of vertices eventually embedded, let $\phi:W\to V(\sigma)$ be the resulting partial embedding, and let $L=\Lplus(B)\cap W$. A vertex $x\in W$ is \emph{active at} $i$ if its image has index at most $i$ and some child of $x$ is unembedded or has image after $v_i$.

We first record two claims that apply to every failed position at which an active vertex exists. Let $\ell_i<i$ be the largest index whose preimage is active at $i$, and let $I_i=\{v_j:\ell_i<j\le i\}$.

\smallskip
\noindent\emph{Claim 1.} For every such $v_i$,
\begin{equation}\label{eq:local-failure-leaf}
 |F\cap I_i|\le |\phi(L)\cap I_i|.
\end{equation}

Indeed, every out-neighbour of $v_{\ell_i}$ in $I_i$ is hit. Otherwise, when $v_{\ell_i}$ was scanned, the greedy rule would have used that position before leaving a child unembedded or placing one after $v_i$. Hence $F\cap I_i\subseteq\Nminus(v_{\ell_i})\cap I_i$. If $v_j\in\Nplus(v_{\ell_i})\cap I_i$, then its preimage is not active at $i$. Moreover, no descendant is unembedded or has image after $v_i$: otherwise, on a path to such a descendant, the last vertex embedded at or before $v_i$ would be active at $i$ and would have image after $v_{\ell_i}$. Since every descendant is placed to the right of $v_j$, the whole descendant subtree is embedded in $I_i$ and contains an out-leaf whose image lies there. These leaves may be chosen distinctly. Indeed, $v_j$ was hit no later than the end of the scan of $v_{\ell_i}$, so the father of its preimage has image at or before $v_{\ell_i}$; consequently, no two of these preimages are comparable in the rooted order. By \eqref{eq:feedback-forward}, $|\Nminus(v_{\ell_i})\cap I_i|\le|\Nplus(v_{\ell_i})\cap I_i|$, which proves the claim.

\smallskip
\noindent\emph{Claim 2.} The intervals $I_i$ arising from failed positions with an active vertex form a laminar family.

Suppose that $i<j$ and the two intervals overlap without one containing the other. Then $\ell_i<\ell_j<i$. The preimage of $v_{\ell_j}$ is not active at $i$, so all its children are already embedded at or before $v_i$; it cannot become active at $j$, a contradiction.

Suppose that the procedure fails to embed all of $B$. Then $|F|\ge k_B$, since at most $n_B-1$ of the $n_B+k_B-1$ positions are hit. Moreover, every failed position has an active vertex. Indeed, fix an unembedded vertex and follow its path from the root; the first vertex that is unembedded or has image after the current failed position has a father embedded at or before that position, and this father is active. Since an unembedded vertex has an unembedded out-leaf descendant, $|L|\le k_B-1$. The maximal intervals in Claim~2 are disjoint and cover $F$; summing \eqref{eq:local-failure-leaf} over them gives $|F|\le|L|\le k_B-1$, a contradiction. Thus the procedure succeeds and leaves exactly $k_B-1$ failed positions.

It remains to construct the injection. Once a failed position occurs with no active vertex, the whole arborescence has already been embedded before that position; hence no later failed position has an active vertex. Thus the failed positions with an active vertex form a prefix of $F$, and those without one form a suffix. Process the first class from left to right. At a failed position $v_i$, Claim~1 supplies an unused leaf image in $I_i$: any leaf image in $I_i$ used earlier was assigned to an earlier failed position that also lies in $I_i$, while the current position is counted in $F\cap I_i$. After these positions have been treated, assign the remaining failed positions from left to right to arbitrary unused out-leaves. All of $B$ precedes the first of them, and there are only $k_B-1$ failed positions altogether, so this is always possible. In every case the assigned leaf image precedes the failed position.

Finally, when a parent is scanned, the set of positions chosen for its children depends only on their number, not on their identities. These positions lie strictly to the right and are still unscanned. Their identities and order may therefore be fixed after the complete set has been exposed but before any member of it is scanned. Once all choices are fixed, the execution is exactly the greedy execution for the resulting labelled arborescence, so the preceding proof applies to its final leaf set.
\end{proof}

A \emph{bi-arborescence} is a rooted oriented tree that is the union of an in-arborescence and an out-arborescence meeting only in their common root. We shall use the following direct consequence of \cref{lem:greedy-arborescence}.

\begin{lemma}\label{lem:bi-arborescence}
Every bi-arborescence on $n\ge2$ vertices is $(2n-2)$-unavoidable.
\end{lemma}

\begin{proof}
If one of the two arborescences consists only of the common root, \cref{lem:greedy-arborescence} and directional duality give $\unv(A)\le n+k-1\le2n-2$. Otherwise, let their orders be $n_I,n_O\ge2$ and their numbers of leaves be $k_I,k_O$. Let $q=n+k_I+k_O-2$. Given a tournament on $2n-2$ vertices, take a local median order and retain its first $q$ positions; this is possible since $k_I\le n_I-1$, $k_O\le n_O-1$, and hence $q\le2n-3$. Place the common root at position $n_I+k_I-1$. The interval ending there has length $n_I+k_I-1$, and the interval beginning there has length $n_O+k_O-1$. Apply the directional dual of \cref{lem:greedy-arborescence} to the first interval and the lemma itself to the second. They meet only at the common root, so the embeddings combine. This also makes explicit the monotonicity from order $q$ to order $2n-2$.
\end{proof}

We next recall the rooted parameters of Dross and Havet. For a rooted oriented tree $(A,r)$, an arc is \emph{upward} if it points away from $r$ and \emph{downward} if it points towards $r$. The nontrivial components of the corresponding forests are denoted by $\cupward_r(A)$ and $\cdown_r(A)$. Define
\[
 \gup_r(A)=\sum_{C\in\cupward_r(A)}(|C|+|\Lplus(C)|-2),\qquad
 \gdown_r(A)=\sum_{C\in\cdown_r(A)}(|C|+|\Lminus(C)|-2).
\]
We use the following facts from~\cite[Section~4, proof of Theorem~4]{DrossHavet2021}.

\begin{lemma}\label{lem:gamma-facts}
Let $A$ be an oriented tree with $n$ vertices and $k$ leaves. For every root $r$, $\gup_r(A)+\gdown_r(A)\le n+k-2$. Moreover, a root minimising $\gdown_r(A)$ may be chosen as a source, and a root minimising $\gup_r(A)$ may be chosen as a sink.
\end{lemma}

For later use, we recall the equivalent-arborescence construction~\cite[Lemma~15]{DrossHavet2021}. Suppose $r$ is a source. For each $C\in\cdown_r(A)$, let $f_C$ be the father of its root, let $n_C=|C|$ and $q_C=|\Lminus(C)|$, delete the arcs of $C$, and make every vertex of $C$, together with $q_C-1$ new vertices, a child of $f_C$. The resulting out-arborescence $A'$ has $n+\sum_C(q_C-1)$ vertices. Its number $K(A')$ of out-leaves satisfies
\begin{equation}\label{eq:equivalent-leaves}
 K(A')\le |\Lplus(A)|+\sum_C\bigl((n_C-q_C)+|C\cap\Lminus(A)|+(q_C-1)\bigr)
 \le k+\sum_C(n_C-1).
\end{equation}
Indeed, an original vertex outside the downward components can be an out-leaf of $A'$ only if it was already an out-leaf of $A$. Inside a component $C$, at most $n_C-q_C$ vertices outside $\Lminus(C)$ become out-leaves; if a vertex of $\Lminus(C)$ becomes an out-leaf, then it has no upward child and is an in-leaf of $A$. The remaining $q_C-1$ possible leaves are artificial. When the complete image set of the children of a vertex is exposed, reserve disjoint blocks of the prescribed sizes $n_C+q_C-1$, one for each downward component $C$ attached there, and use the remaining singleton positions for the unchanged children. In each block, take a local median order of the induced subtournament and apply the directional dual of \cref{lem:greedy-arborescence} to embed $C$, leaving $q_C-1$ positions artificial. The online clause of \cref{lem:greedy-arborescence} permits all these identities to be fixed before any child image is scanned. Restricting the resulting embedding to $A$ gives $\unv(A)\le n+k-1+\gdown_r(A)$.

We also use a consequence of El Sahili's theorem. If $M=(v_1,\ldots,v_m)$ is a median order of a tournament, an embedding $\phi$ of $D$ is an $M$-embedding if every final interval $I$ of $M$ satisfies $|\phi(V(D))\cap I|<|I|/2+1$. El Sahili proved that every oriented tree on $q\ge2$ vertices has such an embedding in every median order of every tournament on $3q-3$ vertices and that adding a forward leaf requires only two additional final positions~\cite{ElSahili2004}.

\begin{lemma}\label{lem:one-sided-extension}
Let $H$ be an oriented tree on $h\ge2$ vertices. If $A$ is obtained from $H$ by successively adding $t$ leaves, each joined by an arc from the old tree to the new vertex, then $\unv(A)\le3h-3+2t$. The same holds with all arcs reversed.
\end{lemma}

\begin{proof}
Take a median order of a tournament on $3h-3+2t$ vertices. Its first $3h-3$ positions form a median order of the induced subtournament, where El Sahili's theorem embeds $H$. This is also an embedding relative to the full order: appending unused final positions only lengthens every relevant final interval and hence weakens the defining inequality. Apply El Sahili's leaf-extension proposition successively, using the next two final positions at each step. Directional duality gives the reversed statement.
\end{proof}

We finish with the leaf-cluster decomposition from~\cite[Section~5, proof of Theorem~6]{DrossHavet2021}. Let $A$ be an oriented tree that is not a bi-arborescence. Its \emph{out-leaf cluster} $S^+$ is defined recursively: every out-leaf belongs to $S^+$, and a vertex with exactly one in-neighbour and all its out-neighbours in $S^+$ also belongs to $S^+$. Define the \emph{in-leaf cluster} $S^-$ dually. Then $S^-\cap S^+=\varnothing$, $A[S^-]$ is a forest of in-arborescences, and $A[S^+]$ is a forest of out-arborescences. The \emph{heart} is $H=A-(S^-\cup S^+)$. Write $h=|H|$, $k_H=|L(H)|$, $s^-=|S^-|$, $s^+=|S^+|$, and $s=s^-+s^+$.

The heart has at least two vertices. Indeed, every component of $A[S^+]$ sends no arc outside $S^+$ and has at most one entering arc, while every component of $A[S^-]$ receives no arc from outside $S^-$ and has at most one leaving arc. By the recursive definitions, any such external arc is incident with the root of the corresponding arborescence component: a non-root vertex already has its unique in-neighbour in an out-component, and dually its unique out-neighbour in an in-component. If $H=\varnothing$ and one cluster is empty, then $A$ is an arborescence. If both clusters are nonempty, the component graph is connected and every component has degree at most one, so it consists of one component of each type joined by a unique arc $u\to w$, where $u$ and $w$ are their respective roots. Absorbing this arc into the out-side gives an out-arborescence rooted at $u$, while the minus-component is an in-arborescence with the same root; hence $A$ is a bi-arborescence. If $H$ has one vertex, any edge from a minus-component to a plus-component would isolate that pair from the heart, so every cluster component is adjacent directly to the heart vertex. Adjoining that vertex to all components of $S^-$ gives an in-arborescence and adjoining it to all components of $S^+$ gives an out-arborescence with the same root, again a bi-arborescence. Every out-leaf of $H$ has an in-neighbour in $S^-$, and every in-leaf of $H$ has an out-neighbour in $S^+$. Hence $s^-\ge|\Lplus(H)|$ and $s^+\ge|\Lminus(H)|$. All leaves of $A$ lie in $S^-\cup S^+$, so if $A$ has $k$ leaves, then $k\le s=n-h$.

\begin{lemma}\label{lem:forward-greedy}
Let $R$ be an out-arborescence on $N\ge2$ vertices, and let $\tau=(w_1,\ldots,w_{2N-2})$ be a local median order of a tournament. Send the root to $w_1$. At each subsequent step, choose an arbitrary unembedded child of any embedded vertex and send it to the first unused out-neighbour of its father's image to the right. The process succeeds for every sequence of such choices. More precisely, after $q\ge2$ vertices have been embedded, if $R_q$ is the current sub-arborescence and $\phi$ the current embedding, then $\phi(V(R_q))\subseteq\{w_1,\ldots,w_{2q-2}\}$ and $\phi$ is root-forward in $\tau_q=(w_1,\ldots,w_{2q-2})$.
\end{lemma}

\begin{proof}
For $q=2$, \eqref{eq:feedback-forward} applied to $(w_1,w_2)$ gives $w_1\to w_2$, and the only proper final interval of $\tau_2$ is $\{w_2\}$. Suppose the assertion holds after $q<N$ vertices have been embedded. Choose the next child arbitrarily, write $\phi$ for the current embedding, and suppose its father is mapped to $w_i$. If $i<2q-2$, root-forwardness applied to the proper final interval $\{w_{i+1},\ldots,w_{2q-2}\}$ gives
\[
 |\phi(V(R_q))\cap\{w_{i+1},\ldots,w_{2q-2}\}|<\frac{2q-2-i}{2}+1=\frac{2q-i}{2}.
\]
The same inequality is immediate when $i=2q-2$, since its left-hand side is zero. By \eqref{eq:feedback-forward} applied to $(w_i,\ldots,w_{2q})$, the vertex $w_i$ has at least $\lceil(2q-i)/2\rceil$ out-neighbours to its right in that interval, whereas the displayed strict inequality and integrality show that at most $\lceil(2q-i)/2\rceil-1$ of those positions are already used. Thus an unused out-neighbour exists, and the first one lies among $w_{i+1},\ldots,w_{2q}$. The enlarged embedding is therefore contained in $\tau_{q+1}$.

After placing the new child, the proper final intervals contained in $\{w_{2q-1},w_{2q}\}$ satisfy the required inequality directly, since at most one of these two positions is used. Every other proper final interval of $\tau_{q+1}$ is obtained from a proper final interval of $\tau_q$ by adding two positions and at most one image, so the strict inequality is preserved. This proves the assertion for $q+1$.
\end{proof}

\section{Recovering the artificial leaves}\label{sec:recovery}

We refine the three-phase embedding in the proof of~\cite[Theorem~6, equations~(3)--(6)]{DrossHavet2021}. Fix a non-bi-arborescence $A$ for which both $S^-$ and $S^+$ are nonempty, and root its heart $H$ at a source $r$. For $C\in\cdown_r(H)$, let $n_C=|C|$ and $q_C=|\Lminus(C)|$, and define
\[
 G=\sum_{C\in\cdown_r(H)}(n_C-1),\qquad U=\sum_{C\in\cdown_r(H)}(q_C-1).
\]
Then $\gdown_r(H)=G+U$. Let $H'$ be the equivalent out-arborescence. It has $h+U$ vertices, including $U$ artificial out-leaves. Every original out-leaf of $H$ remains an out-leaf: it has no child in the rooted heart and cannot be the attachment father of a downward component. If $K'=|\Lplus(H')|$, applying \eqref{eq:equivalent-leaves} to $H$ gives $K'\le k_H+G$. Let $B'$ be obtained from $A[V(H)\cup S^+]$ by replacing $H$ with $H'$; thus $|B'|=h+U+s^+$.

\begin{lemma}\label{lem:fixed-root}
With the notation above,
\begin{equation}\label{eq:fixed-root}
 \unv(A)\le2s+3h-k_H-3+\gdown_r(H)+2|\Lminus(H)|.
\end{equation}
\end{lemma}

\begin{proof}
Set $\ell^-=|\Lminus(H)|$ and $\ell^+=|\Lplus(H)|$, so $k_H=\ell^-+\ell^+$, and define
\begin{equation}\label{eq:buffers}
 \lambda=h-k_H-1+G-U+2\ell^-+2s^-,\qquad
 m=2s+3h-k_H-3+G+U+2\ell^-.
\end{equation}
Here $G-U=\sum_C(n_C-q_C)\ge0$. If $h\ge3$, then $k_H\le h-1$, while if $h=2$, then $k_H=2$ and $\ell^-=1$; since $s^-\ge1$, in either case $\lambda\ge1$. Take a tournament $T$ on $m$ vertices with local median order $\sigma=(v_1,\ldots,v_m)$.

In Phase~1, run the greedy procedure of \cref{lem:greedy-arborescence} on $H'$ in the exact interval $(v_{\lambda+1},\ldots,v_p)$, where $p=\lambda+h+U+K'-1$. Whenever a sibling image set is exposed, recover the corresponding downward component and fix the child identities before any of those positions is scanned. The online clause of \cref{lem:greedy-arborescence} shows that the completed execution is a greedy execution for the final labelled $H'$, so its failed-position injection applies to the actual original and artificial out-leaves. Since the nontrivial out-arborescence $H'$ has $K'\le h+U-1$ and
\[
 m-\lambda=2(h+U+s^+)-2=2|B'|-2,
\]
we have $p\le m$.

In Phase~2, whenever a vertex of the current embedding has an unembedded out-neighbour in $S^+$, choose one whose image has the smallest index and send an arbitrary such child to the first currently unhit out-neighbour of its image to the right. View the child positions exposed simultaneously in Phase~1 in increasing order. Every choice in Phases~1 and~2 is then the first available out-neighbour in the full interval $(v_{\lambda+1},\ldots,v_m)$: during Phase~1, \cref{lem:greedy-arborescence} finds every required position before $v_p$, and every position after $v_p$ has a larger index. The interval $(v_{\lambda+1},\ldots,v_m)$ is itself a local median order by interval heredity. Since it has length $m-\lambda=2|B'|-2$, and since \cref{lem:forward-greedy} allows the next child to be chosen arbitrarily, that lemma shows that all of $B'$ is embedded. We then release the $U$ artificial images. Keeping Phase~1 separate is essential: its shorter endpoint $p$, rather than the endpoint used only to guarantee completion of $B'$, is what yields the later estimate \eqref{eq:hit2-first}.

 For each vertex $c\in V(H)$ having an in-neighbour in $S^-$, let $F_c$ be the union of $c$ and all components of $A[S^-]$ whose root sends its unique external arc to $c$. Every component of $A[S^-]$ has such an attachment in $H$: it has exactly one outgoing external arc, and if this arc entered $S^+$, then the corresponding components of $S^-$ and $S^+$ would form a component of $A-H$ disconnected from $H$. Thus the sets $V(F_c)\setminus\{c\}$ partition $S^-$, and each $F_c$ is an in-arborescence rooted at $c$. 


In Phase~3, order these attachment vertices by decreasing positions of their images. Process them in this order, completing $F_c-c$ before moving to the next attachment vertex. Within the current branch $F_c$, choose any embedded vertex having an unembedded in-neighbour in $F_c$, and send an arbitrary such in-neighbour to the last currently unhit in-neighbour of its image to the left. No position is released after Phase~3 begins.

Suppose that this phase fails while processing the branch rooted at $c$, and write $v_j=\phi(c)$. Then $\lambda+1\le j\le p$. Let $b$ be the vertex at which the failure occurs, with image $v_i$, and let
\[
 b=x_0,x_1,\ldots,x_t=c
\]
be the directed path from $b$ to $c$ in $F_c$; thus $x_r\to x_{r+1}$. Let $v_{i_r}=\phi(x_r)$, so $i=i_0<i_1<\cdots<i_t=j$. Let $\operatorname{hit}$ be the number of hit positions among $v_1,\ldots,v_{j-1}$ at the moment of failure. Since no unhit in-neighbour of $v_i$ remains to its left, \eqref{eq:feedback-backward} gives at least $(i-1)/2$ hit positions in $\{v_1,\ldots,v_{i-1}\}$. For every $0\le r<t$, the image $v_{i_r}$ was chosen as the last unhit in-neighbour of $v_{i_{r+1}}$. Hence every in-neighbour of $v_{i_{r+1}}$ in $\{v_{i_r+1},\ldots,v_{i_{r+1}-1}\}$ was already hit at that time and remains hit. The position $v_{i_r}$ is itself a hit in-neighbour of $v_{i_{r+1}}$, since $x_r\to x_{r+1}$. Applying \eqref{eq:feedback-backward} to $(v_{i_r},\ldots,v_{i_{r+1}})$, we therefore find at least $(i_{r+1}-i_r)/2$ hit positions in $\{v_{i_r},\ldots,v_{i_{r+1}-1}\}$. These intervals are disjoint, and therefore
\begin{equation}\label{eq:chain-hit-lower}
 \operatorname{hit}\ge\frac{j-1}{2}.
\end{equation}

Let $O_{<j}$ and $O_{\ge j}$ be the original out-leaves of $H$ whose images lie in $[\lambda+1,j-1]$ and $[j,p]$, respectively, and define $D_{<j},D_{\ge j}$ analogously for the artificial leaves. All these images were chosen in Phase~1, so
\begin{equation}\label{eq:partitions}
 |O_{<j}|+|O_{\ge j}|=\ell^+,
 \qquad |D_{<j}|+|D_{\ge j}|=U.
\end{equation}
Write $\operatorname{hit}=\operatorname{hit}_2+\operatorname{hit}_3$, according as the current occupant was embedded in Phases~1--2 or Phase~3. Every out-leaf of $H$ is an attachment vertex, and its branch outside $H$ is nonempty. Since the attachment vertices are processed in decreasing image order, each vertex in $O_{<j}$ has an entirely unprocessed branch containing an unembedded vertex of $S^-$. The current branch also contains an unembedded vertex. These witnesses are distinct because the sets $V(F_c)\setminus\{c\}$ partition $S^-$. Consequently,
\begin{equation}\label{eq:hit3}
 \operatorname{hit}_3\le s^--|O_{<j}|-1.
\end{equation}

We use two bounds on $\operatorname{hit}_2$. At the considered moment, the $U$ released artificial images are not Phase~1--2 hits. By \cref{lem:greedy-arborescence}, among the leaves in $O_{\ge j}\cup D_{\ge j}$, all but at most one have distinct failed mates after their images. The original-leaf images, artificial images, and paired failed positions are pairwise distinct. Thus, among the $p-\lambda$ central positions, the $U$ artificial images cannot contribute to $\operatorname{hit}_2$. In addition, the leaves in $O_{\ge j}$ contribute their $|O_{\ge j}|$ images, and all but at most one of the $|O_{\ge j}|+|D_{\ge j}|$ leaves contribute distinct failed mates; hence at least $2|O_{\ge j}|+|D_{\ge j}|-1$ further positions lie at or after $j$. It follows that
\begin{align}
 \operatorname{hit}_2
 &\le p-\lambda-U-2|O_{\ge j}|-|D_{\ge j}|+1\notag\\
 &\le h+k_H+G-2|O_{\ge j}|-|D_{\ge j}|,\label{eq:hit2-first}
\end{align}
where $p-\lambda=h+U+K'-1$ and $K'\le k_H+G$. The artificial images in $D_{<j}$ are also not Phase~1--2 hits, giving
\begin{equation}\label{eq:hit2-second}
 \operatorname{hit}_2\le j-\lambda-1-|D_{<j}|.
\end{equation}

Combining \eqref{eq:chain-hit-lower} with the two preceding bounds and \eqref{eq:hit3}, and then using \eqref{eq:partitions}, we obtain
\begin{align*}
 j-1
 &\le 2\operatorname{hit}\\
 &\le h+k_H+G-2|O_{\ge j}|-|D_{\ge j}|+j-\lambda-1-|D_{<j}|\\
 &\hspace{2em}+2s^--2|O_{<j}|-2,
\end{align*}
so
\begin{align*}
 \lambda
 &\le h+k_H+G-U-2|O_{<j}|-2|O_{\ge j}|+2s^--2\\
 &=h+k_H+G-U-2\ell^++2s^--2\\
 &=h-k_H-2+(G-U)+2\ell^-+2s^-\\
 &=\lambda-1,
\end{align*}
which is impossible. Phase~3 therefore succeeds, and \eqref{eq:buffers} gives \eqref{eq:fixed-root}.
\end{proof}

\begin{remark}\label{rem:one-per-artificial}
If the failed mates of the artificial leaves are not used, the left buffer in \eqref{eq:buffers} must be enlarged from $\lambda$ to $\lambda+U$ while the Phase~2 capacity is kept fixed. Thus the refinement recovers exactly one tournament position for each artificial leaf.
\end{remark}

Rewriting \cref{lem:fixed-root} separates the two directions.

\begin{corollary}\label{cor:directional}
If $r$ is a source of $H$, then $\unv(A)\le2s+3h-3+Q_r^{\downarrow}$, where $Q_r^{\downarrow}=\gdown_r(H)+|\Lminus(H)|-|\Lplus(H)|$. If $r$ is a sink, then $\unv(A)\le2s+3h-3+Q_r^{\uparrow}$, where $Q_r^{\uparrow}=\gup_r(H)+|\Lplus(H)|-|\Lminus(H)|$.
\end{corollary}

\begin{proof}
The first estimate follows from \eqref{eq:fixed-root} and $k_H=|\Lminus(H)|+|\Lplus(H)|$; the second is its directional dual.
\end{proof}

\section{Global bounds}\label{sec:global}

\begin{lemma}\label{lem:heart-exact}
If both leaf clusters of $A$ are nonempty, then
\begin{equation}\label{eq:heart-bound-exact}
 \unv(A)\le2s+3h-3+\left\lfloor\frac{h+k_H-2}{2}\right\rfloor.
\end{equation}
\end{lemma}

\begin{proof}
Choose $r_-$ minimising $\gdown_r(H)$ and $r_+$ minimising $\gup_r(H)$. By \cref{lem:gamma-facts}, they may be chosen as a source and a sink, respectively. For any fixed root $r_0$,
\[
 \gdown_{r_-}(H)+\gup_{r_+}(H)
 \le\gdown_{r_0}(H)+\gup_{r_0}(H)\le h+k_H-2.
\]
The leaf-orientation corrections cancel, so $Q_{r_-}^{\downarrow}+Q_{r_+}^{\uparrow}\le h+k_H-2$. One of these integers is at most $\lfloor(h+k_H-2)/2\rfloor$, and \cref{cor:directional} applies.
\end{proof}

\begin{corollary}\label{cor:heart-bound}
If both leaf clusters are nonempty and $h\ge3$, then
\begin{equation}\label{eq:many-heart}
 \unv(A)\le2s+4h-5=2n+2h-5.
\end{equation}
If $h=2$, then $\unv(A)\le2n$.
\end{corollary}

\begin{proof}
For $h\ge3$, the heart has at most $h-1$ leaves, and \eqref{eq:heart-bound-exact} gives $2s+4h-5$. If $h=2$, then $k_H=2$, and the same formula gives $2s+4=2n$.
\end{proof}

\begin{lemma}\label{lem:one-cluster}
If $A$ is not a bi-arborescence and exactly one of $S^-,S^+$ is nonempty, then $\unv(A)\le2n+h-3$.
\end{lemma}

\begin{proof}
By directional duality assume $S^-=\varnothing$. Starting with $H$, add the vertices of $S^+$ in increasing distance from $H$; each is then a new out-leaf. Apply \cref{lem:one-sided-extension} with $t=s^+=n-h$ to obtain $\unv(A)\le3h-3+2s^+=2n+h-3$.
\end{proof}

The next consequence gives a direct uniform improvement of the many-leaf estimate of Dross and Havet.

\begin{corollary}\label{cor:nk-bound}
Every non-bi-arborescence $A$ on $n$ vertices with $k$ leaves satisfies
\begin{equation}\label{eq:nk-bound}
 \unv(A)\le4n-2k-4.
\end{equation}
Consequently,
\[
 4n-2k-4\le
 \left\lceil\frac92n-\frac52k-\frac92\right\rceil-1.
\]
\end{corollary}

\begin{proof}
Recall that $k\le n-h$. If both clusters are nonempty and $h\ge3$, \cref{cor:heart-bound} gives $\unv(A)\le2n+2h-5\le4n-2k-5$. If both are nonempty and $h=2$, then $k\le n-2$, so $2n\le4n-2k-4$. If exactly one cluster is nonempty, \cref{lem:one-cluster} gives $\unv(A)\le2n+h-3\le3n-k-3\le4n-2k-4$. Finally, a non-bi-arborescence is not a star, so $k\le n-2$, and the difference between the Dross--Havet expression and the right-hand side of \eqref{eq:nk-bound} is $(n-k-1)/2\ge1/2$, proving the last assertion after taking the ceiling.
\end{proof}

\begin{proof}[Proof of \cref{thm:main}]
Let $A$ be an oriented tree on $n\ge2$ vertices. If it is a bi-arborescence, \cref{lem:bi-arborescence} applies, and $2n-2\le\lceil(18n-23)/7\rceil$.

Assume otherwise and retain the leaf-cluster notation. Since $k\le s=n-h$, the few-leaf theorem of Dross and Havet gives
\begin{equation}\label{eq:few-h}
 \unv(A)\le\left\lceil\frac32(n+s)-2\right\rceil
 =\left\lceil3n-\frac32h-2\right\rceil.
\end{equation}
If exactly one cluster is nonempty, \cref{lem:one-cluster} and \eqref{eq:few-h} give
\[
 \unv(A)\le\min\left\{2n+h-3,
 \left\lceil3n-\frac32h-2\right\rceil\right\}
 \le\left\lceil\frac{12n-13}{5}\right\rceil
 \le\left\lceil\frac{18n-23}{7}\right\rceil.
\]
Indeed, the two real affine bounds balance at $h=2(n+1)/5$, where their common value is $(12n-13)/5$; the last inequality holds for $n\ge4$. Every oriented tree on $n\le3$ vertices is a bi-arborescence.

Suppose both clusters are nonempty. If $h\ge3$, combine \eqref{eq:few-h} and \eqref{eq:many-heart}. For $h\le2(n+3)/7$, the latter is at most $(18n-23)/7$; for $h>2(n+3)/7$, the real quantity inside the ceiling in \eqref{eq:few-h} is smaller than $(18n-23)/7$. Hence the desired integer bound follows.

It remains that $h=2$. For $n=4$, \eqref{eq:few-h} gives $\unv(A)\le7$. For $n=5$, \cref{cor:heart-bound} gives $\unv(A)\le10=\lceil67/7\rceil$. For $n\ge6$, the same corollary gives $\unv(A)\le2n\le(18n-23)/7$. This completes the proof.
\end{proof}

\section{Sharing artificial children}\label{sec:sharing}

The equivalent-arborescence construction reserves $q_C-1$ artificial children separately for every downward component $C$. Components with a common father can share a pool. This refinement is not used in \cref{thm:main}, but records a further structural saving.

Fix a rooted heart $(H,r)$ with $r$ a source, and write $z_+=\max\{z,0\}$. For a vertex $x$, let $C_1,\ldots,C_t$ be the downward components whose roots have father $x$, and let $n_i=|C_i|$ and $q_i=|\Lminus(C_i)|$. Define
\begin{equation}\label{eq:dx}
 d_x=\min_{\pi\in S_t}\max_{1\le j\le t}
 \left(q_{\pi(j)}-1-\sum_{i>j}n_{\pi(i)}\right)_+,
\end{equation}
and set $d_x=0$ if there is no such component. Let $D_r=\sum_x d_x$ and $P_r=\sum_{C\in\cdown_r(H)}(|\Lminus(C)|-1)-D_r$. Since $d_x\le\sum_{i=1}^t(q_i-1)$ for every $x$, we have $P_r\ge0$.

\begin{proposition}\label{prop:sharing}
In the setting of \cref{lem:fixed-root},
\begin{equation}\label{eq:sharing-bound}
 \unv(A)\le2s+3h-k_H-3+\gdown_r(H)-2P_r+2|\Lminus(H)|.
\end{equation}
\end{proposition}

\begin{proof}
For each father $x$, flatten all downward components attached to it simultaneously and add only $d_x$ artificial children. Fix an order $C_1,\ldots,C_t$ attaining \eqref{eq:dx}. When the outer greedy procedure exposes all child images of $x$, none has yet been scanned. First assign the unchanged children of $x$ to singleton positions and regard the remaining $\sum_i n_i+d_x$ positions as the common pool for $C_1,\ldots,C_t$. Immediately before recovering $C_j$, at least $\sum_{i\ge j}n_i+d_x\ge n_j+q_j-1$ positions remain unlabelled. Select $n_j+q_j-1$ of them, take a local median order of their induced subtournament, and apply the directional dual of \cref{lem:greedy-arborescence} to recover $C_j$. This occupies $n_j$ positions. After all components are recovered, precisely $d_x$ children remain artificial.

Let $H^*$ be the resulting out-arborescence and let $E_r=\sum_{C\in\cdown_r(H)}(|C|-|\Lminus(C)|)$. Then $|H^*|=h+D_r$. The same leaf count as above, with $D_r$ artificial vertices and at most $|C|-|\Lminus(C)|$ new original out-leaves from each component besides the in-leaves already counted by $k_H$, gives $K^*\le k_H+E_r+D_r$. Use
\[
 \lambda^*=h-k_H-1+E_r+2|\Lminus(H)|+2s^-,\qquad
 m^*=2s+3h-k_H-3+E_r+2D_r+2|\Lminus(H)|.
\]
Here $\lambda^*\ge1$: if $h\ge3$, then $h-k_H-1\ge0$, while if $h=2$, then $k_H=2$, $|\Lminus(H)|=1$, and $s^-\ge1$. Run Phase~1 on the exact interval ending at $p^*=\lambda^*+h+D_r+K^*-1$. Since $K^*\le h+D_r-1$ and $m^*-\lambda^*=2(h+D_r+s^+)-2$, we have $p^*\le m^*$, and \cref{lem:forward-greedy} completes Phase~2 exactly as in the proof of \cref{lem:fixed-root}. Release the $D_r$ artificial images.

Run Phase~3 branch by branch, in decreasing order of the images of the attachment vertices, exactly as in the proof of \cref{lem:fixed-root}. If it failed in the branch rooted at an attachment vertex with image $v_j$, the ancestor-chain argument would give $\operatorname{hit}\ge(j-1)/2$, while
\[
 \operatorname{hit}_3\le s^--|O_{<j}|-1.
\]
The analogues of \eqref{eq:hit2-first} and \eqref{eq:hit2-second} are
\[
 \operatorname{hit}_2\le h+k_H+E_r+D_r-2|O_{\ge j}|-|D_{\ge j}|,
 \qquad
 \operatorname{hit}_2\le j-\lambda^*-1-|D_{<j}|.
\]
Using $|D_{<j}|+|D_{\ge j}|=D_r$ and $|O_{<j}|+|O_{\ge j}|=|\Lplus(H)|$, the same calculation gives
\[
 \lambda^*\le h-k_H-2+E_r+2|\Lminus(H)|+2s^-=\lambda^*-1,
\]
a contradiction. Finally,
$
 \gdown_r(H)=E_r+2\sum_C(|\Lminus(C)|-1)=E_r+2D_r+2P_r,
$
so $m^*$ is the right-hand side of \eqref{eq:sharing-bound}.
\end{proof}

\begin{remark}
The parameter $P_r$ measures genuine sharing: it is positive when positions reserved for one reverse component absorb part of the leaf demand of later components with the same father. Thus a tree close to equality in the coarse heart bound must have little such overlap.
\end{remark}

\section{Alternating double brooms}\label{sec:brooms}

We finish with a family illustrating local slack in the heart accounting. It is not a model for the balancing range in the proof of \cref{thm:main}: its heart has order asymptotic to one half of the tree, where the few-leaf estimate is already stronger. The point is instead that a direct two-sided pairing can remove the apparent loss completely.

For nonnegative integers $a,b$, let $B_{a,b}$ have vertices $u,v$, $x_i,y_i$ for $1\le i\le a$, and $z_j,w_j$ for $1\le j\le b$, with arcs $u\to v$, $u\to y_i$, $x_i\to y_i$, $z_j\to v$, and $z_j\to w_j$. Its order is $n_{a,b}=2+2a+2b$.

\begin{proposition}\label{prop:double-broom}
For all $a,b\ge0$, $\unv(B_{a,b})\le2n_{a,b}-2$.
\end{proposition}

\begin{proof}
Take a local median order $(v_1,\ldots,v_{4a+4b+2})$ of a tournament on $2n_{a,b}-2$ vertices, and send $u$ and $v$ to $v_{4b+1}$ and $v_{4b+2}$. Consecutive positions form a forward arc. There are $4a+1$ positions to the right of the image of $u$, at least $2a+1$ of which it dominates. Excluding the image of $v$, choose $2a$ and partition them into pairs; in each pair, name the tail $x_i$ and the head $y_i$. Dually, at least $2b+1$ of the $4b+1$ positions to the left of the image of $v$ dominate it. Excluding the image of $u$, choose $2b$, pair them, and name the tail and head $z_j,w_j$. The two collections lie on opposite sides of the central arc and are disjoint.
\end{proof}

\section{Concluding remarks}

The coefficient $18/7$ comes from using the failed-position pairing for the full leaf set of the equivalent arborescence after the online child relabelling has been completed. The resulting bound \eqref{eq:nk-bound} uniformly improves the previous many-leaf estimate for every non-bi-arborescence, rather than only at the final balancing point. The sharing parameter in \cref{prop:sharing} identifies another source of savings when several reverse components leave the same vertex.

To get closer to $2n$ within this framework, one needs a further linear saving when the heart has order close to $2n/7$, the actual balancing range of the two global estimates. The double-broom argument is not an extremal example for that range, but it shows that simultaneous two-sided pairing can be much more efficient than charging adjacent alternating branches independently.

\section*{Acknowledgements}
The first author was supported by the National Natural Science Foundation of China (Nos.~12522117 and 12401456) and the Fundamental Research Funds for the Central Universities, Nankai University.

\end{document}